\documentclass[11pt]{amsart}

\usepackage{amsmath,amssymb,amsthm}
\usepackage[pdftex]{color,graphicx}
\usepackage[utf8]{inputenc} 
\usepackage{graphics}
\usepackage{enumerate}
\usepackage{array}
\usepackage[english]{babel}
\usepackage{verbatim}
\usepackage{hyperref}
\usepackage{caption}
\usepackage{tikz-cd}

\usepackage{pdflscape}
\usepackage{rotating}
\usepackage{graphicx}

\usepackage{amsmath}
\usepackage{amsfonts}
\usepackage{amssymb}
\usepackage{amsthm}
\usepackage{mathrsfs}
\usepackage{mathtext}
\usepackage{mathtools}

\usepackage[all,knot]{xy}

\newcommand{\ot}{\otimes}

\DeclareMathOperator{\HH}{HH}
\DeclareMathOperator{\HHom}{Hom}

\numberwithin{equation}{section}

\newtheorem{defi}[equation]{Definition}

\newtheorem{theo}[equation]{Theorem}

\newtheorem{lemm}[equation]{Lemma}
\newtheorem{lemma}[equation]{Lemma}

\title[Homotopy liftings and Hochschild cohomology]{Homotopy liftings and Hochschild cohomology\\
of some twisted tensor products}

\author[P.\ S.\ Ocal]{Pablo S.\ Ocal}
\address{Department of Mathematics, Texas A\&M University, 
College Station, Texas 77843, USA}
\email{pso@math.tamu.edu}
\author[T.\ Oke]{Tolulope Oke}
\address{Department of Mathematics, Texas A\&M University, 
College Station, Texas 77843, USA}
\email{toluoke@math.tamu.edu}
\author[S.\ Witherspoon]{Sarah Witherspoon}
\address{Department of Mathematics, Texas A\&M University, 
College Station, Texas 77843, USA}
\email{sjw@math.tamu.edu}

\date{January 24, 2021}

\thanks{Partially supported by NSF grants 1665286 and 2001163.}
\thanks{Partially supported by NSF grant 1440140 while the first  author was in residence at the Mathematical Sciences Research Institute in Berkeley, California, during the Spring 2020 semester.}


\begin{document}


\begin{abstract}
The Hochschild cohomology of a tensor product of algebras is
isomorphic to a graded tensor product of
Hochschild cohomology algebras, as a Gerstenhaber algebra.
A similar result holds when the tensor product is twisted
by a bicharacter. 
We present new proofs of these isomorphisms, using Volkov's homotopy liftings
that were introduced 
for handling Gerstenhaber brackets
expressed on arbitrary bimodule resolutions.
Our results illustrate the utility of  homotopy liftings  for theoretical purposes.
\end{abstract}

\maketitle

\section{Introduction}\label{sec:introduction}

Let $A$ and $B$ be algebras over a field $k$. Let $\HH^*$ denote Hochschild cohomology. In this paper we give a new proof of the isomorphism of Gerstenhaber algebras, 
\begin{equation*}
\HH^*(A\ot B)\cong \HH^*(A)\ot \HH^*(B)
\end{equation*}
(see Theorem~\ref{thm:G-alg-iso}). 
This isomorphism was proven by Le and Zhou \cite{LZ} using Alexander-Whitney and Eilenberg-Zilber maps on bar resolutions to handle the Gerstenhaber bracket structure. 
In fact, we give a new proof of the more general isomorphism of Gerstenhaber algebras,
\begin{equation*}
\HH^{*,F'\oplus G'}(A\ot^t B)\cong \HH^{*,F'}(A)\ot \HH^{*,G'}(B)
\end{equation*}
(see Theorem~\ref{thm:tG-alg-iso}), for a twisted tensor product algebra $A\ot^t B$ where the twisting $t$ comes from a bicharacter on grading groups (notation defined in Section~\ref{sec:twisted}). This isomorphism was proven by Grimley, Nguyen, and the third author~\cite{GNW} using twisted versions of the Alexander-Whitney and Eilenberg-Zilber maps, 
building on a result of Bergh and Oppermann~\cite{BO} about the associative algebra structure. 

For our proofs, we use Volkov's homotopy lifting technique~\cite{Volkov},
designed for use with arbitrary projective resolutions. 
In comparison with proofs already
in the literature, we do not use  bar resolutions and
thus we do not need to use the unwieldy Alexander-Whitney and Eilenberg-Zilber maps. 
These results illustrate the theoretical usefulness of homotopy liftings.
In some settings they are also computationally useful: 
see, for example, \cite{G,GNW} for some quantum complete intersections and
smash products with groups, \cite{KMOOW} for the Jordan plane, and 
\cite{NW} for polynomial rings and some types of
cyclic group algebras. 
In particular, in~\cite{GNW,KMOOW}, elementary methods allow the application of the techniques in~\cite{NW} to compute Gerstenhaber brackets on the Hochschild cohomology of 
twisted tensor products. 
This method relies on the construction of resolutions for twisted tensor product algebras given in~\cite{RTTP}. 
More specifically, the  quantum complete intersections in~\cite{GNW} are algebras twisted by bicharacters, thus providing a large class of examples to which our results
in this paper apply. 
In the last section, 
we illustrate by finding explicitly a homotopy lifting for a small such example, 
a truncated polynomial ring.

\section{Twisted tensor product and Gerstenhaber bracket techniques}\label{sec:twisted}

In this section we summarize definitions, techniques, and results from~\cite{BO,GNW,RTTP,Volkov} on twisted tensor products by a bicharacter, resolutions, and homotopy liftings. The results that we mention here were generalized in \cite{KMOOW} for a strongly graded twist using the bar and the Koszul resolutions, but we remain focused here on the special case of twistings by a bicharacter.

{\em Throughout this paper, all algebras are over a field $k$.}
The use of the tensor product $\otimes$ without any decorations usually means $\otimes_k$, but in Section \ref{sec:tp} some of the computations are carried out where we tensor over a different ring. We have included a warning before that, and we believe the context allows the reader to unequivocally determine the ring over which the tensor products happen.

We will use the Koszul sign convention. Whenever $V$, $W$, $V'$, $W'$ are graded vector spaces and $g: V\rightarrow V'$, $h: W\rightarrow W'$ are graded $k$-linear maps, we define the graded $k$-linear map $g\otimes h: V\otimes V'\rightarrow W\otimes W'$ by
\[
   (g\ot h) (v\ot w) := (-1)^{|h||v|} g(v)\ot h(w)
\]
for all homogeneous $v\in V$, $w\in W$, where $|h|$, $|v|$ denote the degrees of $h$, $v$, respectively. As a consequence, it can be immediately checked that the same sign rule applies to morphisms:
\begin{equation}\label{eqn:fcn-commutes}
   (g\ot h) (g'\ot h') = (-1)^{|h||g'|} (g g')\ot (h h').
\end{equation}

\medskip

\noindent
{\bf{Twisted tensor product by a bicharacter}}.
We now recall the definition of the twisted tensor product of algebras by a bicharacter in the sense of Bergh and Oppermann~\cite{BO}: Let $A$ and $B$ be algebras over the field $k$, graded by groups $F$ and $G$ respectively, and let $t:F\otimes_{\mathbb{Z}} G\rightarrow k^{\times}$ be a homomorphism of abelian groups, also called a \emph{twisting}. We write  $t^{\langle f| g\rangle} = t(f\otimes_{\mathbb{Z}} g)$
for all $f\in F$, $g\in G$. Let $A\otimes^t B$ denote the {\em twisted tensor product} of algebras, that is $A\ot B$ as a vector space with multiplication given by
\begin{equation*}
(a\otimes b)\cdot^t (a'\otimes b') := t^{\langle |a'|||b|\rangle}aa'\otimes bb'
\end{equation*}
for homogeneous $a,a'\in A$ and $b,b'\in B$ of degrees $|a|,|a'|\in F$ and $|b|,|b'|\in G$. We will usually write $t^{\langle a'|b\rangle}$ instead of $t^{\langle |a'|||b|\rangle}$. Observe that $A\otimes^t B$ is $(F\oplus G)$-graded.

Similarly, if $M$ is an $F$-graded $A^e$-module and $N$ is a $G$-graded $B^e$-module, denote by $M\ot^t N$ the $(A\otimes^t B)^e$-module given as a vector space by $M\otimes N$ and module structure given by
\begin{equation}\label{eqn:MoN-tmod-struc}
(a\otimes b)(m\otimes n)(a'\otimes b') := t^{\langle m|b\rangle} t^{\langle a'|n\rangle} t^{\langle a'|b\rangle} ama'\otimes bnb'
\end{equation}
for homogeneous $a,a'\in A$, $b,b'\in B$, $m\in M$, and $n\in N$. 
It can be checked that if 
$M$ and $N$ are projective modules, then $M\ot^t N$ is a $(F\oplus G)$-graded projective $(A\otimes^t B)$-module.

\quad

\noindent {\bf{Twisted tensor product of resolutions}}. 
Let $P$ be a projective resolution of $A$ as an $A^e$-module, and let $Q$ be a projective resolution of $B$ as a $B^e$-module. 
We will assume that for all $i$, $P_i$ is a finitely generated $A^e$-module and $Q_i$ is a finitely generated $B^e$-module or that at least one of $A$ or $B$
is finite dimensional as a vector space over $k$.
These hypotheses ensure that $\HHom$ and $\oplus$ may be interchanged
at a crucial step in the proof of the isomorphism. 
We will consider the projective resolution $P\ot Q$ of $A\ot B$ as an $(A\ot B)^e$-module, and more generally we will consider the projective resolution $P\ot^t Q$ of $A\ot^t B$ as an $(A\ot^t B)^e$-module. In the latter case, the authors of~\cite{GNW} constructed several isomorphisms of modules that can be seen as a chain map between two resolutions of $A\otimes^t B$ as $(A\otimes^t B)^e$-module, as showcased in Lemma~\ref{lemm:sigma-chain-iso} below. For this, they required the resolutions $P$ and $Q$ to be free as $A^e$- and $B^e$-modules respectively, as well as $P_0 = A\otimes A$ and $Q_0 = B\otimes B$, so we shall also assume these additional conditions.

\begin{lemm}\label{lemm:sigma-chain-iso} \cite[Lemma 3.2]{GNW}
There is a chain map
\begin{equation*}
\sigma:(P\otimes^t Q)\otimes_{A\otimes^t B}(P\otimes^t Q)\rightarrow (P\otimes_A P)\otimes^t (Q\otimes_B Q)
\end{equation*}
that is an isomorphism of $(A\otimes^t B)^e$-modules in each degree, given by
\begin{equation*}
\sigma((x\otimes y)\otimes (x'\otimes y')) := (-1)^{ju}t^{\langle x' | y \rangle} (x\otimes x')\otimes (y\otimes y')
\end{equation*}
on $(P_i\otimes^t Q_j)\otimes_{A\otimes^t B}(P_u\otimes^t Q_v)$.
\end{lemm}

Clearly $\sigma^{-1}:(P\otimes_A P)\otimes^t (Q\otimes_B Q)\rightarrow (P\otimes^t Q)\otimes_{A\otimes^t B}(P\otimes^t Q)$ is given by
\begin{equation*}
\sigma^{-1}((x\otimes x')\otimes (y\otimes y')) = (-1)^{uj}t^{-\langle x' | y \rangle} (x\otimes y)\otimes (x'\otimes y')
\end{equation*}
on $(P_i\otimes_A P_u)\otimes^t(Q_j\otimes_B Q_v)$.

\quad

\noindent
{\bf{Tensor product of Gerstenhaber algebras}}.
We now recall the definition of a graded tensor product of two Gerstenhaber algebras 
from Manin~\cite[Chapter V, Proposition 9.11.1]{Manin} 
(cf.\ Le and Zhou~\cite[Remark~2.3(2) and Proposition-Definition~2.2]{LZ}, which differs by signs from that which we will use here).
Let $H_1$, $H_2$ be two Gerstenhaber algebras. Let $f,f'\in H_1$ and $g,g'\in H_2$ be homogeneous elements of degrees $m,m',n,n'$, respectively. Then the graded vector space $H_1\ot H_2$ is a Gerstenhaber algebra with product
\begin{equation}\label{eqn:cup-Galg}
   (f\ot g ) \smile (f'\ot g') := (-1)^{m'n} (f\smile f') \ot (g\smile g')
\end{equation}
and bracket
\begin{equation}\label{eqn:bracket-Galg}
{[} f\ot g , f'\ot g' {]} := (-1)^{(m'-1) n} {[}f,f'{]} \ot (g\smile g') + (-1)^{m'(n-1)} (f\smile f') \ot {[} g, g'{]} .
\end{equation}
This is exactly the definition of a graded tensor product of two Gerstenhaber--Batalin--Vilkovisky algebras (GBV algebras) given by Manin~\cite[Chapter~V Proposition~9.11.1]{Manin}. At this point we must exercise caution, since~\cite{LZ} omits the first name when referring to these algebras and calls them Batalin--Vilkovisky algebras (BV algebras), in alignment with the nomenclature by Getzler~\cite{Getzler}. Moreover, as observed in \cite[Definition~2.4]{LZ}, Batalin-Vilkovisky algebras are a special case of Gerstenhaber algebras.

\quad

\noindent
{\bf{Hochschild cohomology of a twisted tensor product}}.
Next we recall a standard isomorphism on Hochschild cohomology rings: Due to the finiteness hypothesis on $P$ and $Q$,
there is an isomorphism of vector spaces for each $m$, $n$,
\[
  \HHom_{(A\ot B)^e} (P_m\ot Q_n , A\ot B) \cong
    \HHom_{A^e}(P_m,A) \ot \HHom_{B^e}(Q_n,B) .
\]
These isomorphisms give rise to an isomorphism of complexes
\[
  \HHom_{(A\ot B)^e}(P\ot Q, A\ot B)\cong \HHom_{A^e}(P,A)\ot\HHom_{B^e}(Q,B) ,
\]
which in turn induces the standard isomorphism of 
associative algebras
\[
  \HH^*(A\ot B) \cong \HH^*(A)\ot \HH^*(B) .
\]
At this point we note that this isomorphism is in fact an isomorphism of Gerstenhaber algebras, a theorem of Le and Zhou, 
stated as Theorem~\ref{thm:G-alg-iso} below. 
More generally, as noted in~\cite[Remarks 6.4]{GNW}, when taking into account grading by groups $F,G$ and twisting by a bicharacter $t$, this isomorphism of associative algebras in fact restricts to the subalgebras
\[
  \HH^{*,F'\oplus G'}(A\ot^t B) \cong \HH^{*,F'}(A)\ot \HH^{*,G'}(B)
\]
where $F'$ and $G'$ are subgroups of $F$ and $G$, defined respectively by 
\begin{equation}\label{eqn:FprimeGprime}
F' := \bigcap_{u\in G} \text{Ker}\; t^{\langle -| u\rangle} \text{ and }  G' := \bigcap_{v\in F} \text{Ker}\; t^{\langle v| -\rangle}.
\end{equation}
The indicated second grading on Hochschild cohomology is that induced by the grading of
$A,B$, by $F,G$, respectively. 
We restrict to these subalgebras of the Hochschild cohomology algebras  
because the interchange of $\mathrm{Hom}$ and $\otimes$ does not behave well with respect to graded bimodules and degree shifts. (See~\cite[Remark 4.2 and Theorem 4.7]{BO} for details.)

Again, this isomorphism above is in fact an isomorphism of Gerstenhaber algebras, a theorem of Grimley, Nguyen, and the last author~\cite{GNW}, stated as Theorem~\ref{thm:tG-alg-iso} below.

Note that $f\in\HHom_{A^e}(P_m,A)$ is a cocycle representing a class in $\HH^{m,v}(A)$ where $v\in F'$ if for all homogeneous $x\in P_m,\;|f(x)|=|x|-v$. Similarly $g\in\HHom_{A^e}(Q_n,A)$ is a cocycle representing a class in $\HH^{n,u}(B)$ for $u\in G'$, if for all homogeneous $y\in Q_n,\;|g(y)|=|y|-u$. Hence
\begin{align}\label{eqn:t-F'-G'}
t^{\langle f(x)-x|y\rangle} = t^{\langle |f(x)|-|x|{|}y\rangle} = t^{\langle -v|y\rangle} = t^{-\langle v |y\rangle} = 1,\notag\\
t^{\langle x | g(y)-y\rangle} = t^{\langle x | |g(y)|-|y|\rangle} = t^{\langle x | -u\rangle} = t^{-\langle x | u\rangle} = 1.
\end{align}

\quad

\noindent {\bf{Homotopy lifting}\label{sec:Gbracket-techniques}}.
Next we summarize techniques from~\cite{NW,Volkov},
as reformulated in~\cite[Section 6.3]{HCSW}, 
for understanding and computing Gerstenhaber brackets on the Hochschild
cohomology ring $\HH^*(R)$ of any $k$-algebra $R$. 

The graded Lie algebra structure on the Hochschild cohomology ring $\HH^*(R)$
is defined on the bar complex, with equivalent recent definitions on
other resolutions. 
In this paper we take the formula in~(\ref{eqn:formula-bracket}) below
to be our definition of the Gerstenhaber bracket, and refer to the
cited literature for details on equivalent definitions. 

Let $P\stackrel{\mu}{\rightarrow} A$ be a projective resolution of $A$ as an $A^e$-module with differential $d$ and augmentation map $\mu:P_0\rightarrow A$. We take \textbf{d} to be the differential on the $\HHom$ complex $\HHom_{A^e}(P,P)$ defined for all $A^e$-maps $f:P\rightarrow P[-m]$ as
$$\textbf{d}(f) := df - (-1)^m fd.$$

In the following definition, the notation $\sim$ is used for two
cohomologous cocycles, that is, differing by a coboundary.

\begin{defi}\label{defi:bracket}
{\em Let $K\stackrel{\mu}{\rightarrow} R$ be a projective resolution of $R$ as an $R^e$-module, let $\Delta: K\rightarrow K\ot_R K$ be a chain map lifting the identity map on $R$, and let $f\in\HHom_{R^e}(K_m,R)$ be a cocycle. An $R^e$-module homomorphism $\psi_f:K\rightarrow K[1-m]$ is a {\em homotopy lifting} of $f$ with respect to $\Delta$ if 
\begin{align}\label{eqn:hl}
\textbf{d}(\psi_f) &= (f\ot 1 - 1\ot f)\Delta \quad \text{and} \\
\mu\psi_f &\sim (-1)^{m-1}f\psi \notag
\end{align}
for some $\psi:K\rightarrow K[1]$ for which $\textbf{d}(\psi) = (\mu\ot 1 - 1\ot \mu)\Delta.$ }
\end{defi}

We will make heavy use of the following theorem of Volkov. 

\begin{theo}\label{thm:bracket} \cite{Volkov}
Let $K$ be a projective resolution of $R$ as an $R^e$-module. 
Let $f\in {\rm{Hom}}_{R^e}(K_m,R)$ and $g\in{\rm{Hom}}_{R^e}(K_n,R)$ be cocycles on $K$, and let $\psi_f$ and $\psi_g$ be homotopy liftings of $f$ and $g$ , as in Definition~\ref{defi:bracket}. Then
\begin{equation}\label{eqn:formula-bracket}
   {[} f , g {]} := f\psi_g - (-1)^{(m-1)(n-1)} g \psi_f
\end{equation}
is a function in ${\rm{Hom}}_{R^e}(K_{m+n-1},R)$ representing the
Gerstenhaber bracket on Hochschild cohomology at the chain level. 
\end{theo}
\begin{proof} See~\cite{Volkov} or~\cite[Section 6.3]{HCSW}.
\end{proof}

\section{Isomorphisms of Gerstenhaber algebras}\label{sec:tp}

In this section, we give a new proof of a result of Grimley, Nguyen, and the third author~\cite{GNW}: A particular subalgebra of the bigraded Hochschild cohomology of a twisted tensor product by a bicharacter $A\ot^t B$ of algebras $A$ and $B$ (under a finiteness condition) is isomorphic, as a Gerstenhaber algebra, to a subalgebra of the graded tensor product of the bigraded Hochschild cohomology algebras of $A$ and $B$, with bracket given by formula~(\ref{eqn:bracket-Galg}). The proof in~\cite{GNW} used bar resolutions combined with twisted versions of the Alexander-Whitney and Eilenberg-Zilber maps. Here we dispense with bar resolutions altogether and give a direct proof independent of choices of projective resolutions, illustrating the theoretical value of Volkov's homotopy lifting method~\cite{Volkov}. 

Our proof will immediately translate to the case without the bigrading, showing that the bracket given by formula~(\ref{eqn:bracket-Galg}) on the graded tensor product $\HH^*(A)\ot \HH^*(B)$ of the two Gerstenhaber algebras $\HH^*(A)$ and $\HH^*(B)$ corresponds to the Gerstenhaber bracket on $\HH^*(A\ot B)$. Although this is a special case of the general theorem by taking the grading or the twisting to be trivial, it is enlightening to see how the proof does not rely on the particulars of the grading nor the twisting, which suggests that it could be improved to the setting discussed in~\cite{KMOOW}.

We will use Volkov's theory of homotopy liftings~\cite{Volkov} as summarized in Section~\ref{sec:Gbracket-techniques}. We will first find the homotopy liftings needed for the proof, in Lemma~\ref{lem:twist-fotg-hl} below. Let $\Delta_P: P\rightarrow P\ot_A P$ and $\Delta_Q:Q\rightarrow Q\ot_B Q$ be diagonal maps induced by chain maps $\tilde{\Delta}_P:P\rightarrow P\ot P$ and $\tilde{\Delta}_Q:Q\rightarrow Q\ot Q$ lifting the multiplication maps $A\ot A\rightarrow A$ and $B\ot B\rightarrow B$, respectively. Let $$\Delta_{P\ot^t Q} := \sigma^{-1}(\Delta_P \ot^t \Delta_Q) ,$$ where $\sigma$ is the isomorphism given in Lemma~\ref{lemm:sigma-chain-iso}, and so $\Delta_{P\ot^t Q}$ is by construction a diagonal map from $P\ot^t Q$ to $(P\ot^t Q)\ot_{A\ot^t B}(P\ot^t Q)$. Moreover $\Delta_{P\ot^t Q}$ is clearly induced by the chain map $$\tilde{\Delta}_{P\ot^t Q} := {\widetilde{\sigma}}^{-1}(\tilde{\Delta}_P \ot \tilde{\Delta}_Q)$$ 
where $\widetilde{\sigma}^{-1}: (P\ot P) \ot (Q\ot Q) \rightarrow
(P\ot Q) \ot (P\ot Q)$ is defined by the ``same'' formula as $\sigma^{-1}$ (after the statement of Lemma~\ref{lemm:sigma-chain-iso}). 

Let $f\in\HHom_{A^e}(P_m,A)$, $f'\in\HHom_{A^e}(P_{m'},A)$, $g\in\HHom_{B^e}(Q_n,B)$, $g'\in\HHom_{B^e}(Q_{n'}, B)$ be cocycles representing elements of $H_1:= \HH^{*,F'}(A)$ and $H_2:= \HH^{*,G'}(B)$. Denote by $f\ot^t g$ the element of $\HHom_{(A\ot^t B)^e}(P_m\ot^t Q_n, A\ot^t B)$ defined by $(f\ot^t g)(x\ot y) = (-1)^{mn} t^{-\langle x | g\rangle} f(x)\ot g(y)$ for all $x\in P_m$, $y\in Q_n$, and similarly $f'\ot^t g'$. It can be checked that $f\ot^t g$, $f'\ot^t g'$ are indeed $(A\ot^t B)^e$-module homomorphisms due to the definitions of the subgroups $F',G'$ of $F,G$. Furthermore, $f\ot^t g$ and $f'\ot^t g'$ are cocycles due to the definition of the differential on the tensor product of complexes.

\begin{lemma}\label{lem:twist-fotg-hl}
Let $\psi_f,\psi_g$ be homotopy liftings of $f,g$ with respect to $\Delta_P$, $\Delta_Q$, respectively, and define
\[
   \psi_{f\ot^t g}:= \psi_f \ot (1\ot g)\Delta_Q + (-1)^m (f\ot 1)
    \Delta_P \ot \psi_g 
\]
as an element of $\HHom_{(A\ot^t B)^e}(P\ot ^t Q, P\ot^t Q[1-m-n])$.
Then $\psi_{f\ot^t g}$ is a homotopy lifting of $f\ot^t g$
with respect to $\Delta_{P\ot^t Q}$.
\end{lemma}

Note that we are using a slight abuse of notation: In the definition of $\psi_{f\ot^t g}$ the first tensor symbol should be $\ot^t$, the second should be $\ot_B$, the third should be $\ot_A$, and the fourth should be $\ot^t$. However this should be clear from the domains and codomains of the maps used. When it does not cause any confusion, this abuse of notation will carry on in the proofs.

\begin{proof}
First note that $\psi_{f\ot^t g}$ is indeed an $(A\ot^t B)^e$-module
homomorphism as claimed.

We will next show that $\psi_{f\ot^t g}$ satisfies 
equation~(\ref{eqn:hl}) with $f$ replaced by $f\ot^t g$, that is,
we will show that
\begin{equation}\label{eqn:twist-fotg}
   d \psi_{f\ot^t g} - (-1)^{m+n-1} \psi_{f\ot^t g} d = 
    ((f\ot^t g)\ot 1_{P\ot^t Q} - 1_{P\ot^t Q}\ot (f\ot^t g)) \Delta_{P\ot^t Q} .
\end{equation} 
Applying the definition of $\psi_{f\ot^t g}$ given in the
statement of the lemma, the left side of equation~(\ref{eqn:twist-fotg}) is 
\[
\begin{aligned}
 & d (\psi_f\ot (1\ot g)\Delta_Q) + (-1)^m d ((f\ot 1)\Delta_P\ot \psi_g) \\
&  \quad +(-1)^{m+n} (\psi_f \ot (1\ot g)\Delta_Q ) d 
   + (-1)^n ((f\ot 1)\Delta_P\ot \psi_g)d \\
& = d\psi_f \ot (1\ot g)\Delta_Q + (-1)^{m-1} \psi_f \ot d(1\ot g)\Delta_Q
   + (-1)^m d (f\ot 1)\Delta_P \ot \psi_g\\
& \quad + (f\ot 1)\Delta_P\ot d\psi_g 
   +(-1)^m \psi_f d\ot (1\ot g)\Delta_Q + (-1)^{m+n} \psi_f \ot (1\ot g)
  \Delta_Q d \\
& \quad - (f\ot 1)\Delta_P d \ot \psi_g
   + (-1)^{n} (f\ot 1)\Delta_P\ot\psi_g d .
\end{aligned}
\]
The maps 
$(f\ot 1)\Delta_P$ and $(1\ot g)\Delta_Q$ 
commute with the differentials $d$ up to multiplication by
$(-1)^m$ and $(-1)^n$, respectively.
Thus four of the above terms cancel, leaving
\[
 (d\psi_f - (-1)^{m-1} \psi_f d)\ot (1\ot g)\Delta_Q
   + (f\ot 1)\Delta_P \ot (d\psi_g - (-1)^{n-1}\psi_g d) ,
\]
in which we use equation~(\ref{eqn:hl}) for $f$ and for $g$ to obtain
\[
   (f\ot 1 - 1\ot f)\Delta_P\ot (1\ot g)\Delta_Q
   + (f\ot 1) \Delta_P \ot (g\ot 1-1\ot g)\Delta_Q .
\] 
The right hand side of equation~(\ref{eqn:twist-fotg}) acts on a general element of $P\otimes^t Q$, and applying $\Delta_P \ot^t \Delta_Q$ gives a finite sum of elements of the form $(x\otimes x')\otimes (y\otimes y')\in (P\otimes_A P)\otimes^t (Q\otimes_B Q)$. We can then compute this right hand side as
\begin{align*}
 &((f\ot^t g)\ot 1_{P\ot^t Q} - 1_{P\ot^t Q}\ot (f\ot^t g))\sigma^{-1}((x\otimes x') \otimes (y\otimes y'))\\
 &\quad = ((f\ot^t g)\ot 1_{P\ot^t Q} - 1_{P\ot^t Q}\ot (f\ot^t g))(-1)^{|x'||y|} t^{-\langle x' | y \rangle}(x\otimes y \otimes x'\otimes y')\\
 &\quad = (-1)^{|x'||y|} t^{-\langle x' | y \rangle} (-1)^{n|x|}f(x)\ot g(y)\ot x'\ot y'\\
 &\quad - (-1)^{|x'||y|} t^{-\langle x' | y \rangle} (-1)^{n|x|}(-1)^{n|y|}(-1)^{n|x'|}(-1)^{m|x|}(-1)^{m|y|} x\ot y \ot f(x')\ot g(y') .
\end{align*}
This can only be nonzero when applied to elements whose degrees coincide with the degrees of $f$ and $g$, that is, we can assume that $m = |x| = |x'|$ and $n = |y| = |y'|$. Hence the computation simplifies to
\begin{equation*}
t^{-\langle x' | y \rangle} f(x)\ot g(y)\ot x'\ot y' - t^{-\langle x' | y \rangle} (-1)^{n|y|}(-1)^{m|x|} x\ot y \ot f(x')\ot g(y') .
\end{equation*}
Projecting onto $P\ot^t Q$ (that is, applying the module action \eqref{eqn:MoN-tmod-struc}) yields
\begin{align}\label{eqn:induced-projection}
 &t^{-\langle x' | y \rangle} t^{\langle x' | g(y)\rangle}f(x)x'\ot g(y)y' - t^{-\langle x' | y \rangle} t^{\langle f(x') | y \rangle} (-1)^{n|y|+m|x|}xf(x')\ot yg(y')\nonumber\\
 &\quad = t^{\langle x' | g(y)-y\rangle}f(x)x'\ot g(y)y' - t^{\langle f(x')-x' | y \rangle} (-1)^{n|y|+m|x|}xf(x')\ot yg(y')\nonumber\\
 &\quad = f(x)x'\ot g(y)y' - (-1)^{n|y|+m|x|}xf(x')\ot yg(y') ,
\end{align}
where the last equality holds because $f$ and $g$ represent cocycles in $\HH^{*,F'}(A)$ and $\HH^{*,G'}(B)$ respectively, so equalities \eqref{eqn:t-F'-G'} apply.

Consider now the behavior of $((f\ot 1)\ot^t (g\ot 1) - (1\ot f)\ot^t (1\ot g))$ applied to an element of the form $(x\otimes x')\otimes (y\otimes y')\in (P\otimes_A P)\otimes^t (Q\otimes_B Q)$. We obtain
\begin{align*}
 &((f\ot 1)\ot^t (g\ot 1) - (1\ot f)\ot^t (1\ot g))(x\otimes x' \otimes y\otimes y')\\
 &\quad = (-1)^{n|x|} (-1)^{n|x'|} f(x)\ot x'\ot g(y)\ot y'\\
 &\quad - (-1)^{m|x|}(-1)^{n|x|}(-1)^{n|x'|}(-1)^{n|y|} x\ot f(x')\ot y \ot g(y')\\
 &\quad = f(x)\ot x'\ot g(y)\ot y' - (-1)^{n|y|+m|x|} x\ot f(x')\ot y \ot g(y')
\end{align*}
where we have again assumed that $m = |x| = |x'|$ and $n = |y| = |y'|$ for the last equality. Projecting onto $P\ot^t Q$ (that is, applying the module action of $A$ on $P$ and $B$ on $Q$) yields $f(x)x'\ot g(y)y' - (-1)^{n|y|+m|x|}xf(x')\ot yg(y')$, which is exactly what we obtained in \eqref{eqn:induced-projection}. This means that $((f\ot^t g)\ot 1_{P\ot^t Q} - 1_{P\ot^t Q}\ot (f\ot^t g))\sigma^{-1}(\Delta_P \ot^t \Delta_Q)$ and $((f\ot 1)\ot^t (g\ot 1) - (1\ot f)\ot^t (1\ot g))(\Delta_P \ot^t \Delta_Q)$ yield the same map after projecting onto $P\ot^t Q$. Since these canonical projections that we use are isomorphisms, we can safely work as if they were equal. In particular since by Lemma~\ref{lemm:sigma-chain-iso} the map $\mu_{P\ot^t Q}$ can be identified with $\mu_P\ot^t \mu_Q$, the above argument shows that $(\mu_P\ot 1\ot\mu_Q\ot 1 - 1\ot\mu_P\ot 1 \ot\mu_Q)(\Delta_P\ot \Delta_Q)$ and $(\mu_{P\ot Q}\ot 1_{P\ot Q} - 1_{P\ot Q} \ot \mu_{P\ot Q})(\Delta_P\ot \Delta_Q)$ can be regarded as equal, a fact we will use in \eqref{eqn:mu-application} below. We now have
\[
\begin{aligned}
& ((f\ot 1)\ot^t (g\ot 1) - (1\ot f)\ot^t (1\ot g))(\Delta_P \ot^t \Delta_Q)\\
&= (f\ot 1\ot g\ot 1 -f\ot 1\ot 1\ot g + f\ot 1\ot 1\ot g 
      - 1\ot f\ot 1\ot g)
    (\Delta_P \ot \Delta_Q)\\
& = ((f\ot 1)\ot (g\ot 1-1\ot g))(\Delta_P \ot \Delta_Q)
  + ((f\ot 1 - 1\ot f)\ot (1\ot g))(\Delta_P \ot \Delta_Q) ,
\end{aligned}
\]
which agrees with what we calculated above for the left
side of equation~(\ref{eqn:twist-fotg}).

Next we take $\psi_P: P\rightarrow P[1]$ and $\psi_Q:Q\rightarrow Q[1]$
to be maps for which $d\psi_P + \psi_P d = (\mu_P\ot 1 - 1\ot \mu_P)\Delta_P$,
$d\psi_Q + \psi_Q d = (\mu_Q\ot 1 -1\ot \mu_Q)\Delta_Q$,
$\mu_P\psi_f \sim (-1)^{m-1} f \psi_P$, and $\mu_Q\psi_g\sim (-1)^{n-1}f\psi_Q$. 
Set
\[
    \psi_{P\ot^t Q} := \psi_P\ot (\mu_Q\ot 1)\Delta_Q + (1\ot\mu_P)\Delta_P\ot\psi_Q.
\]
Then, since $(\mu_Q\ot 1)\Delta_Q$ and $(1\ot\mu_P)\Delta_P$
are chain maps, and noting $|\mu_Q| = 0 = |\mu_P|$,
\begin{align}\label{eqn:mu-application}
& d\psi_{P\ot^t Q} + \psi_{P\ot^t Q}d = d(\psi_P\ot (\mu_Q\ot 1)\Delta_Q + (1\ot\mu_P)\Delta_P\ot \psi_Q)\nonumber\\
 & \hspace{.1cm}+ (\psi_P\ot (\mu_Q\ot 1) \Delta_Q + (1\ot \mu_P)\Delta_P \ot \psi_Q)d\nonumber\\
 &= d\psi_P\ot (\mu_Q\ot 1)\Delta_Q - \psi_P\ot d(\mu_Q\ot 1)\Delta_Q
 + d(1\ot \mu_P)\Delta_P\ot\psi_Q \nonumber\\
 & \hspace{.1cm}+ (1\ot\mu_P)\Delta_P\ot d\psi_Q + \psi_P d \ot (\mu_Q\ot 1)\Delta_Q + \psi_P\ot (\mu_Q\ot 1) \Delta_Q d \nonumber \\
 & \hspace{.1cm}-(1\ot\mu_P)\Delta_P d \ot \psi_Q + (1\ot\mu_P)\Delta_P\ot \psi_Q d\nonumber\\
 &= (d\psi_P + \psi_Pd) \ot (\mu_Q\ot 1)\Delta_Q
   + (1\ot\mu_P)\Delta_P \ot (d\psi_Q + \psi_Q d)\nonumber\\
 &= (\mu_P\ot 1 - 1\ot\mu_P)\Delta_P \ot (\mu_Q\ot 1) \Delta_Q 
   + (1\ot \mu_P)\Delta_P \ot (\mu_Q\ot 1 - 1\ot\mu_Q)\Delta_Q\nonumber\\
 &= (\mu_P\ot 1 \ot\mu_Q\ot 1 -1\ot\mu_P\ot\mu_Q\ot 1\nonumber\\
 &\hspace{.1cm}+ 1\ot \mu_P\ot \mu_Q \ot 1 - 1\ot \mu_P\ot 1 \ot\mu_Q) (\Delta_P
  \ot \Delta_Q)\nonumber\\
&= (\mu_P\ot 1\ot\mu_Q\ot 1 - 1\ot\mu_P\ot 1 \ot\mu_Q)(\Delta_P\ot \Delta_Q)\\
&= (\mu_{P\ot Q}\ot 1_{P\ot Q} - 1_{P\ot Q} \ot \mu_{P\ot Q})\Delta_{P\ot^t Q}\nonumber\\
&= (\mu_{P\ot^t Q}\ot 1_{P\ot^t Q} - 1_{P\ot^t Q} \ot \mu_{P\ot^t Q})\Delta_{P\ot^t Q} , \nonumber
\end{align}
where the second to last equality has already been discussed. Finally we check:
\begin{align*}
  \mu_{P\ot^t Q} \psi_{f\ot^t g} &= (\mu_P\ot\mu_Q)
   (\psi_f\ot (1\ot g)\Delta_Q + (-1)^m (f\ot 1)\Delta_P\ot \psi_g) \\
  &= \mu_P\psi_f \ot\mu_Q (1\ot g)\Delta_Q + (-1)^m \mu_P (f\ot 1)
   \Delta_P \ot\mu_Q \psi_g \\
  &\sim (-1)^{m-1} f\psi_P \ot \mu_Q (1\ot g)\Delta_Q 
  + (-1)^{m+n-1} \mu_P (f\ot 1) \Delta_P \ot g\psi_Q\\
  &= (-1)^{m-1} f\psi_P\ot g (\mu_Q\ot 1)\Delta_Q + (-1)^{m+n-1}
   f(1\ot\mu_P) \Delta_P\ot g\psi_Q\\
  &= (f\ot g) ( (-1)^{m+n-1} \psi_P\ot (\mu_Q\ot 1)\Delta_Q 
   + (-1)^{m+n-1} (1\ot \mu_P) \Delta_P\ot \psi_Q) \\
  &= (-1)^{m+n-1}(f\ot g) (\psi_P\ot (\mu_Q\ot 1)\Delta_Q
    + (1\ot \mu_P)\Delta_P\ot \psi_Q)\\
  &= (-1)^{m+n-1} (f\ot g)\psi_{P\ot^t Q} .
\end{align*}
We have again identified $\mu_{P\ot^t Q}$ with $\mu_P\ot^t \mu_Q$, and used both $\mu_Q(1\otimes g)\Delta_Q = g(1\otimes \mu_Q)\Delta_Q$ and $\mu_P(f\otimes 1)\Delta_P = f(1\otimes \mu_P)\Delta_Q$. Here we justify the first of these equalities, while the second can be checked in an analogous way: These maps act on a general element of $Q$, where applying $\Delta_Q$ gives a finite sum of elements of the form $y\otimes y'\in Q\otimes_B Q$, and now
\begin{align*}
  \mu_Q(1\otimes g)(y\otimes y') &= (-1)^{n|y|}\mu_Q(y\otimes g(y')) = (-1)^{n|y|}\mu_Q(yg(y')\otimes 1)\\
  &= (-1)^{n|y|}\mu_Q(yg(y')) = (-1)^{n|y|}\mu_Q(y)g(y') = g(\mu_Q(y)y')\\
  &= g(1\otimes\mu_Q(y)y') = g(\mu_Q(y)\otimes y') = g(\mu_Q\otimes 1)(y\otimes y') , 
\end{align*}
since the canonical projections are isomorphisms.

In summary, we have shown that
\[
   \mu_{P\ot^t Q} \psi_{f\ot^t g} \sim (-1)^{m+n-1} (f\ot^t g)
   \psi_{P\ot^t Q}
\]
and $d\psi_{P\ot^t Q} + \psi_{P\ot^t Q}d = (\mu_{P\ot^t Q}\ot 1_{P\ot Q} - 1_{P\ot Q}\ot\mu_{P\ot^t Q})\Delta_{P\ot^t Q}$. Therefore $\psi_{f\ot^t g}$ is a homotopy lifting for $f\ot^t g$ with respect to $\Delta_{P\ot^t Q}$.
\end{proof}

Next we state and give the promised new proof of the 
isomorphism of Gerstenhaber algebras.

\begin{theo}[Grimley--Nguyen--Witherspoon~\cite{GNW}]\label{thm:tG-alg-iso}
Let $A$ and $B$ be algebras over a field $k$, graded by groups $F$ and $G$ respectively. Assume that there exist projective resolutions of $A$ as an $A^e$-module and of $B$ as a $B^e$-module consisting of finitely generated modules
or that at least one of $A$ or $B$ is finite dimensional as a vector space over $k$.  
Then there is an isomorphism of Gerstenhaber algebras
\[
   \HH^{*,F'\oplus G'}(A\ot^t B) \cong \HH^{*,F'}(A)\ot \HH^{*,G'}(B) ,
\]
where $F'$ and $G'$ are the subgroups of $F$ and $G$ defined in~\eqref{eqn:FprimeGprime} 
and the algebra on the right side is a graded tensor product 
with product and bracket given by formulas~(\ref{eqn:cup-Galg})
and~(\ref{eqn:bracket-Galg}).
\end{theo}

\begin{proof}
Let $f, f'\in\HH^{*,F'}(A)$ and $g, g'\in\HH^{*,G'}(B)$, be cocycles. By definition~(\ref{eqn:bracket-Galg}) and under the isomorphism of graded vector spaces given by sending $f\ot^t g$ to $f\ot g$, the bracket  ${[}f\ot^t g, f'\ot^t g'{]}$ on the graded tensor product  $\HH^{*,F'}(A)\ot \HH^{*,G'}(B)$ corresponds to
\begin{equation}\label{eqn:Ga}
[f\ot^t g,f'\ot^t g'] = (-1)^{(m'-1)n} {[}f,f'{]}\ot^t (g\smile g') + (-1)^{m' (n-1)} (f\smile f')\ot^t {[}g,g'{]}.
\end{equation}
We may take $f\smile f' = (f'\ot f)\Delta_P$ and $g\smile g'=(g'\ot g)\Delta_Q$, and we may take $[f,f']$ and $[g,g']$ to be given in terms of homotopy liftings by formula~(\ref{eqn:formula-bracket}). Then the right side of~(\ref{eqn:Ga}) is equal to
\begin{align*}
&(-1)^{(m'-1)n} (f\psi_{f'} - (-1)^{(m-1)(m'-1)} f'\psi_f ) \ot^t (g\ot g') \Delta_Q \\
& + (-1)^{m' (n-1)} (f\ot f')\Delta_P \ot^t (g\psi_{g'} - 
   (-1)^{(n-1)(n'-1)} g' \psi_g ).
\end{align*}

By Lemma~\ref{lem:twist-fotg-hl}, a homotopy lifting map for $f\ot^t g$ is
\[  \psi_{f\ot^t g} :=  \psi_{f}\ot (1\ot g) \Delta_Q + (-1)^{m}  (f\ot 1) \Delta_P \ot \psi_{g} . 
\]
Define $ \psi_{f'\ot^t g'}$ similarly for $f'\ot^t g'$. Using formula~(\ref{eqn:formula-bracket}) for $[f\ot^t g,f'\ot^t g']$, the Gerstenhaber bracket via homotopy liftings in $\HH^{*,F'\oplus G'}(A\ot^t B)$, we have
\begin{align*}
[f\ot^t g,f'\ot^t g'] &= (f\ot^t g) \psi_{f'\ot^t g'} - (-1)^{(m+n-1)(m'+n'-1)} (f'\ot^t g')\psi_{f\ot^t g} \\
& = (f\ot^t g) (\psi_{f'}\ot (1\ot g') \Delta_Q) + (-1)^{m'} (f\ot^t g)
  ((f'\ot 1) \Delta_P \ot \psi_{g'}) \\
& \quad - (-1)^{(m+n-1)(m'+n'-1)} (f'\ot^t g') (\psi_f\ot (1\ot g)
  \Delta_Q ) \\
 & \quad - (-1)^{(m+n-1)(m'+n'-1) +m}
   (f'\ot^t g')((f\ot 1)\Delta_P \ot \psi_g).
\end{align*}
Now $f (f' \otimes 1) \Delta_P = (-1)^{mm'}(f' \otimes f) \Delta_P$ and $g (1 \otimes g') \Delta_Q = (g \otimes g') \Delta_Q$ by~(\ref{eqn:fcn-commutes}), so the above becomes
\begin{align}\label{eqn:Gb-hhAtB}
&(-1)^{(m'-1)n} (f\psi_{f'} \ot^t (g\ot g')\Delta_Q) 
   + (-1)^{m'+m'n+mm'} (f'\ot f)\Delta_P\ot^t g\psi_{g'} \notag\\
& \quad - (-1)^{(m+n-1)(m'+n'-1)+(m-1)n'} f'\psi_f \ot^t (g'\ot g)\Delta_Q \notag\\
   & \quad - (-1)^{(m+n-1)(m'+n'-1) +m +mn' +mm'} (f\ot f')\Delta_P\ot^t g'\psi_g .
\end{align}
Now $ (f' \otimes f) \Delta_P \sim (-1)^{mm'} (f \otimes f') \Delta_P$ and $ (g' \otimes g) \Delta_Q \sim (-1)^{nn'} (g \otimes g') \Delta_Q$ because the cup product is graded commutative: Since $f' \smile f $ and $ (-1)^{mm'}  f \smile f'$ differ by a coboundary we have
\[
(f' \ot f)\Delta_P = f' \smile f \sim (-1)^{|f||f'|}  f \smile f' =   (-1)^{mm'} (f \ot f')\Delta_P ,
\]
and similarly for $(g \otimes g') \Delta_Q$. So what we obtained in \eqref{eqn:Gb-hhAtB} differs by a coboundary from
\begin{align*}
&(-1)^{(m'-1)n } f\psi_{f'} \ot^t (g\ot g')\Delta_Q
   - (-1)^{(m'-1)(m+n-1)} f'\psi_f \ot^t (g\ot g')\Delta_Q \\
& \quad + (-1)^{m'(n-1)} (f\ot f') \Delta_P \ot^t g\psi_{g'} 
  - (-1)^{(n-1)(m'+n'-1)} (f\ot f')\Delta_P  \ot^t g'\psi_g \\
& =   (-1)^{(m'-1)n} (f\psi_{f'} - (-1)^{(m-1)(m'-1)} f'\psi_f )\ot^t (g\ot g')\Delta_Q \\
& \quad +(-1)^{m'(n-1)} (f\ot f')  \Delta_P\ot^t (g\psi_{g'} - (-1)^{(n-1)(n'-1)} g'\psi_g) .
\end{align*}
This is equal to the expression~(\ref{eqn:Ga}) found before. We thus conclude that the two bracket expressions agree  in cohomology.
\end{proof}

Considering now the case without the bigrading, the setup is completely analogous with two exceptions:
First, instead of defining $\Delta_{P\ot^t Q}$ we define $\tilde{\Delta}_{P\ot Q} := (1\ot \tau\ot 1)\tilde{\Delta}_P \ot \tilde{\Delta}_Q$, where $\tau$ is the graded flip map (i.e.~$\tau(x\ot y) = (-1)^{|x||y|} y\ot x$ for all homogeneous $x\in P$ and $y\in Q$). Then $\tilde{\Delta}_{P\ot Q}$ is a chain map lifting the multiplication map $(A\ot B)\ot (A\ot B)\rightarrow A\ot B$ on the tensor product algebra $A\ot B$. Let $\Delta_{P\ot Q}$ be the induced diagonal map from $P\ot Q$ to $(P\ot Q)\ot_{A\ot B}(P\ot Q)$, and second, we let $f\in\HHom_{A^e}(P_m,A)$, $f'\in\HHom_{A^e}(P_{m'},A)$, $g\in\HHom_{B^e}(Q_n,B)$, $g'\in\HHom_{B^e}(Q_{n'}, B)$ be cocycles representing elements of $H_1:= \HH^*(A)$ and $H_2:= \HH^*(B)$.

\begin{lemma}\label{lem:fotg-hl}
Let 
\[
   \psi_{f\ot g}:= \psi_f \ot (1\ot g)\Delta_Q + (-1)^m (f\ot 1)
    \Delta_P \ot \psi_g . 
\]
Then $\psi_{f\ot g}$ is a homotopy lifting of $f\ot g$
with respect to $\Delta_{P\ot Q}$.
\end{lemma}

\begin{proof}
Note that in the proof of Lemma~\ref{lem:twist-fotg-hl} the bicharacter does not appear at the end of the computations, and it also would not appear in any of the canonical projections we use. Thus that proof holds taking $F=1,G=1,t=1$, 
since $\tau(\Delta_P\ot \Delta_Q)$ behaves exactly like $\sigma^{-1}(\Delta_P\ot^t \Delta_Q)$ with the only difference that in the former the bicharacter does not appear.
\end{proof}

\begin{theo}[Le--Zhou~\cite{LZ}]\label{thm:G-alg-iso}
Let $A$ and $B$ be algebras over the field $k$. Assume that there exist projective resolutions of $A$ as an $A^e$-module and of $B$ as a $B^e$-module consisting of finitely generated modules or that at least one of $A$ or $B$
is finite dimensional as a vector space over $k$.
Then there is an isomorphism of Gerstenhaber algebras
\[
   \HH^*(A\ot B) \cong \HH^*(A)\ot \HH^*(B) ,
\]
where the algebra on the right side is a graded tensor product 
with product and bracket given by formulas~(\ref{eqn:cup-Galg})
and~(\ref{eqn:bracket-Galg}).
\end{theo}

\begin{proof}
Note that the proof of Theorem~\ref{thm:tG-alg-iso} holds, taking $F = F'= 1$, $G = G'= 1$, $t$ the trivial bicharacter, and the homotopy lifting of Lemma~\ref{lem:fotg-hl}.
\end{proof}

We want to remark that, although we can see Lemma~\ref{lem:fotg-hl} and Theorem~\ref{thm:G-alg-iso} as special cases of Lemma~\ref{lem:twist-fotg-hl} and Theorem~\ref{thm:tG-alg-iso}, we need not do so. We want to emphasize that the formal expressions of the homotopy liftings of the two lemmas are identical, as well as the formal expression (\ref{eqn:bracket-Galg})
of the Gerstenhaber bracket to which each theorem refers. We also want to emphasize that the formal computations that needed to be carried out to prove both lemmas and both theorems are identical.

Our results in this paper 
not only showcase the utility of Volkov's homotopy lifting techniques on a theoretical level, but also suggest that finding them in practice can be a manageable task, 
as we show in the next section.

\section{Example}

We illustrate our results by finding  homotopy liftings for some Hochschild
cocycles in a small example:
the truncated polynomial ring $k[x,y]/(x^2,y^2)$ where $k$ is a field of
characteristic~0. We then compute a Gerstenhaber bracket (cf.~\cite[\S5.2]{GNW}). 

Let $A=k[x]/(x^2)$, $B=k[y]/(y^2)$. Let $P$ be the following projective resolution of $A$:
\begin{equation}\label{p-res}
P: \qquad\cdots \xrightarrow{ \cdot v } A\ot A \xrightarrow{ \cdot u  }A\ot A \xrightarrow{ \cdot v }A\ot A \xrightarrow{ \cdot u }A\ot A\;(\;\xrightarrow{\mu_P} A)
\end{equation}
where $u = x\ot 1 - 1\ot x$, $v = x\ot 1+1\ot x$ and the augmentation map $\mu_P$ is multiplication. For each $i$, let $e_i$ denote the element $1\ot 1$ in $A\ot A$ in degree $i$ and set $e_i=0$ whenever $i<0$. Analogously, define $Q$ to be the projective resolution of $B$ defined in a similar way to $P$ above with $u' = y\ot 1 - 1\ot y$, $v' = y\ot 1+1\ot y$, and free basis element of $Q_i:= B\ot B$ denoted $e_i'$ for each $i$. 
A diagonal map $\Delta_P:P\xrightarrow{} P\ot_A P$ can be defined by
\begin{equation*}
\Delta_P(e_i) = \sum_{j+l=i}  e_j\ot e_l
\end{equation*}
(cf.~\cite[Example 4.7.1]{NVW} where there is a slightly different sign convention). 
We define a diagonal map $\Delta_Q$ for the resolution $Q$ analogously. 

Similar to~\cite[Example 2.2.2]{HCSW} (or see \cite[Example 4.7.1]{NVW}), 
we now consider the Hochschild 1-cocycle $f:P\xrightarrow{}A$ defined by $f(e_1)=x, f(e_i)=0$ for $i\neq 1$ and the Hochschild 2-cocycle $g:Q\xrightarrow{}B$ given by $g(e_2')=y, g(e_i')=0$ for $i\neq 2$. We show that the maps $f_i:P_i\xrightarrow{} P_i$ defined by
$$f_i(e_i) = ie_i$$
and $g_j:Q_j\xrightarrow{}Q_{j-1}$ defined by
$$g_{2j}(e_{2j}') = e_{2j-1}', \quad g_{2j-1}(e_{2j-1}') = 0$$
are homotopy lifting maps of $f$ and $g$ respectively (cf.~\cite[Example 4.7.1]{NVW}). Since these algebras are Koszul, we will only need to show that the first part of Equation \eqref{eqn:hl} holds. When we consider $f$, the right hand side of \eqref{eqn:hl} applied to $e_i$ is
\begin{align*}
(f\ot1 - 1\ot f)\Delta_P(e_i) &= (f\ot1 - 1\ot f) \sum_{j+l=i} e_j\ot e_l\\
&=(f\ot 1)(e_1\ot e_{i-1}) - (1\ot f)(e_{i-1}\ot e_{1})\\ &= xe_{i-1}+ (-1)^{i} e_{i-1}x ,
\end{align*}
and the left hand side of Equation \eqref{eqn:hl} applied to $e_i$ is
\begin{align*}
(uf_i - f_{i-1}u)(e_i) &= u(ie_i) - f_{i-1}(xe_{i-1}-e_{i-1}x)\\
&=i(xe_{i-1}-e_{i-1}x) - x(i-1)e_{i-1}+(i-1)e_{i-1}x\\ & = xe_{i-1}-e_{i-1}x
\end{align*}
whenever $i$ is odd and 
\begin{align*}
(vf_i - f_{i-1}v)(e_i) &= v(ie_i) - f_{i-1}(xe_{i-1}+e_{i-1}x)\\
&=i(xe_{i-1}+e_{i-1}x) - x(i-1)e_{i-1}-(i-1)e_{i-1}x\\ & = xe_{i-1}+e_{i-1}x
\end{align*}
  whenever $i$ is even. So we see that Equation \eqref{eqn:hl} holds for $f$.  In a similar fashion, when we consider $g$, the right hand side of \eqref{eqn:hl} applied to $e_{2i}'$ is
\begin{align*}
(g\ot1 - 1\ot g)\Delta_Q(e_{2i}') &= (g\ot1 - 1\ot g) \sum_{j+l=2i} e_j'\ot e_l'\\
&=(g\ot 1)(e_2'\ot e_{2i-2}') - (1\ot g)(e_{2i-2}'\ot e_{2}')\\
& = ye_{2i-2}'-  e_{2i-2}'y ,
\end{align*}
and the left hand side of Equation \eqref{eqn:hl} applied to $e_{2i}'$ is
\begin{align*}
&(u'g_{2i} + g_{2i-1}v')(e_{2i}') = u'(e_{2i-1}') + g_{2i-1}(ye_{2i-1}'+e_{2i-1}'y) =ye_{2i-2}'-e_{2i-2}'y.
\end{align*}
Again when we consider $g$, the right hand side of \eqref{eqn:hl}
applied to $e_{2i-1}'$ is
\begin{align*}
&(g\ot1 - 1\ot g)\Delta_Q(e'_{2i-1}) = (g\ot1 - 1\ot g) \sum_{j+l=2i-1}  e_j'\ot e_l'\\
&=(g\ot 1)(e_2'\ot e_{2i-3}') - (1\ot g)(e_{2i-3}'\ot e_{2}') = ye_{2i-3}'-  e_{2i-3}'y
\end{align*}
and the left hand side of Equation \eqref{eqn:hl} applied to $e_{2i-1}'$ is
\begin{align*}
&(v'g_{2i-1} + g_{2i-2}u')(e_{2i-1}') = v'(0) + g_{2i-2}(ye_{2i-2}'-e_{2i-2}'y) =ye_{2i-3}'-e_{2i-3}'y.
\end{align*}
So Equation \eqref{eqn:hl} holds for $g$. 
Thus the maps $f_i$ and $g_i$ defined above constitute homotopy lifting maps
$\psi_f$ and $\psi_g$ for $f$ and $g$, respectively. 

The truncated polynomial ring $k[x,y]/(x^2,y^2)$ is isomorphic to the 
tensor product $A\ot B$ of the algebras $A$ and $B$. 
Let $P\ot Q\xrightarrow{} A\ot B$ be the tensor product resolution. 
Let $f\ot g$ be a representative of the class in $\HH^{*}(A\ot B)$ 
corresponding to the tensor product of $f$ and $g$.

Lemma \ref{lem:fotg-hl} gives an expression for a homotopy lifting $\psi_{f\ot g}: P\ot Q\xrightarrow{} P\ot Q[-2]$. For example, applied to $e_1\ot e_2'$, 
\begin{align*}
\psi_{f\ot g}(e_1\ot e_2') &= (f_{*}\ot(1\ot g)\Delta_Q - (f\ot 1)\Delta_P\ot g_{*})(e_1\ot e_2')\\
&= f_{1}(e_1)\ot(1\ot g)(e_0'\ot e_2' + e_1'\ot e_1' + e_2'\ot e_0') + (f\ot 1)(e_0\ot e_1 + e_1\ot e_0)\ot g_{2}(e_2')\\
&= e_1\ot(1\ot g)( e_0'\ot e_2') + (f\ot 1)(e_1\ot e_{0})\ot e_{1}'\\
&= e_1\ot e_{0}'y +  xe_{0}\ot e_{1}' . 
\end{align*}
Alternatively, it can be verified directly that $\psi_{f\ot g}$ as defined
above is a homotopy lifting, that is it satisfies Equation \eqref{eqn:twist-fotg}. 

Similarly, we may let $h:P\rightarrow A$ be defined by $h(e_2)=1$, 
$h(e_i)=0$ for $i\neq 2$, a Hochschild 2-cocycle \cite[Example 2.2.2]{HCSW}.
We may check that $(h\ot 1 - 1\ot h)\Delta_P(e_i)=0$ for all $i$,
and so we may take the zero function as a homotopy lifting of $h$
with respect to $\Delta_P$. 
Take similarly $f', h': Q\rightarrow B$ defined by
$f'(e_1')=y$, $f'(e_i')=0$ for $i\neq 1$,
and $h'(e_2')=1$, $h'(e_i')=0$ for $i\neq 2$.
The zero function is a homotopy lifting for $h'$, and
$f_i': Q_i\rightarrow Q_i$ defined by $f_i'(e_i')=ie_i'$ constitutes a
homotopy lifting for $f'$.
Similarly to the above calculations, using Lemma \ref{lem:fotg-hl} we find
\begin{eqnarray*}
  \psi_{f\ot f'}(e_2\ot e_3') &=& 2e_2\ot e_2'y - 3 xe_1\ot  e_3' , \\
  \psi_{f\ot f'} (e_3\ot e_2') & = & 3e_3\ot e_1y - xe_2\ot 2e_2' .
\end{eqnarray*}
Thus by formula \eqref{eqn:formula-bracket}, we obtain a nonzero bracket:
\begin{eqnarray*}
   {[} f\ot f' , h\ot h' {]} (e_2\ot e_3') & = &
      (f\ot f') \psi_{h\ot h'} ( e_2\ot e_3') + (h\ot h') \psi_{f\ot f'}(e_2\ot e_3')\\
    &=& (h\ot h')(2e_2\ot e_2'y - 3xe_1\ot e_3') \\
     &=&  2y , \\
  {[}f\ot f', h\ot h' {]} (e_3\ot e_2') & = & (h\ot h')\psi_{f\ot f'} (e_3\ot e_2')\\     
       &=& (h\ot h') ( 3e_3\ot e_1'y - xe_2\ot 2e_2') \\
   & = & -2x .
\end{eqnarray*}
Compare with \cite[\S5.2]{GNW} where Gerstenhaber brackets of all generators
of Hochschild cohomology of $k[x,y]/(x^2,y^2)$ were found with considerably
greater effort. An advantage there however is the additional information
on brackets with elements in degree~0, something not addressed currently in the
homotopy lifting theory.

This small example shows that the homotopy lifting method is not only a
valuable theoretical tool, but also can be used for explicit computations
of Gerstenhaber brackets via formula \eqref{eqn:formula-bracket}.

\end{document}